\documentclass[12pt]{amsart}

\usepackage{amsmath}
\usepackage{amscd}
\usepackage{amssymb}
\usepackage{latexsym}
\usepackage{color}
\usepackage{mathrsfs}
\usepackage{hyperref}

\newtheorem{theorem}{Theorem}[section] 
\newtheorem*{theorem*}{Theorem}
\newtheorem{lemma}[theorem]{Lemma}
\newtheorem*{lemma*}{Lemma}
\newtheorem{corollary}[theorem]{Corollary}
\newtheorem*{corollary*}{Corollary}
\newtheorem{proposition}[theorem]{Proposition}
\newtheorem*{proposition*}{Proposition}

\newtheorem{remark}[theorem]{Remark}
\newtheorem{question}[theorem]{Question}
\newtheorem{definition}[theorem]{Definition}

\newtheorem{example}[theorem]{Example}
\newcommand{\bgl}{\begin{equation}} 
\newcommand{\egl}{\end{equation}}
\newcommand{\bgloz}{\begin{equation*}} 
\newcommand{\egloz}{\end{equation*}}
\newcommand{\bgln}{\begin{eqnarray}} 
\newcommand{\egln}{\end{eqnarray}}
\newcommand{\bglnoz}{\begin{eqnarray*}} 
\newcommand{\eglnoz}{\end{eqnarray*}}
\newcommand{\btheo}{\begin{theorem}}
\newcommand{\etheo}{\end{theorem}}
\newcommand{\btheooz}{\begin{theorem*}}
\newcommand{\etheooz}{\end{theorem*}}
\newcommand{\blemma}{\begin{lemma}}
\newcommand{\elemma}{\end{lemma}}
\newcommand{\blemmaoz}{\begin{lemma*}}
\newcommand{\elemmaoz}{\end{lemma*}}
\newcommand{\bproof}{\begin{proof}}
\newcommand{\eproof}{\end{proof}}
\newcommand{\bbew}{\begin{beweis}}
\newcommand{\ebew}{\end{beweis}}
\newcommand{\bremark}{\begin{remark}\em}
\newcommand{\eremark}{\end{remark}}
\newcommand{\bquestion}{\begin{question}\em}
\newcommand{\equestion}{\end{question}}
\newcommand{\bdefin}{\begin{definition}}
\newcommand{\edefin}{\end{definition}}
\newcommand{\bprop}{\begin{proposition}}
\newcommand{\eprop}{\end{proposition}}
\newcommand{\bpropoz}{\begin{proposition*}}
\newcommand{\epropoz}{\end{proposition*}}
\newcommand{\bcor}{\begin{corollary}}
\newcommand{\ecor}{\end{corollary}}
\newcommand{\bcoroz}{\begin{corollary*}}
\newcommand{\ecoroz}{\end{corollary*}}
\newcommand{\bfa}{\begin{cases}} 
\newcommand{\efa}{\end{cases}}
\newcommand{\bexample}{\begin{example}\em}
\newcommand{\eexample}{\end{example}}
\newcommand{\cA}{\mathcal A}

\newcommand{\cD}{\mathcal D}

\newcommand{\cG}{\mathcal G}

\newcommand{\cK}{\mathcal K}

\newcommand{\cM}{\mathcal M}
\newcommand{\cN}{\mathcal N}
\newcommand{\cO}{\mathcal O}

\newcommand{\cU}{\mathcal U}

\def\Cz{\mathbb{C}}

\def\Nz{\mathbb{N}}

\def\Rz{\mathbb{R}}
\def\Tz{\mathbb{T}}
\def\Zz{\mathbb{Z}}

\def\1z{\mathbb{1}}

\newcommand{\an}[1]{``#1''} 
\newcommand{\ti}{\tilde}

\newcommand{\lori}{\longrightarrow}

\newcommand{\ma}{\mapsto} 
\newcommand\onto{\twoheadrightarrow} 
\newcommand\into{\hookrightarrow} 

\newcommand{\ve}{\varepsilon}

\def\SEMI{\mbox{$\times\kern-2pt\vrule height5pt width.6pt \kern3pt $}}

\newcommand{\Hom}{{\rm Hom}\,}

\newcommand{\Aut}{{\rm Aut}\,}

\newcommand{\Spec}{{\rm Spec\,}} 

\newcommand{\id}{{\rm id}}

\newcommand{\Ad}{{\rm Ad\,}}

\newcommand{\ev}{\operatorname{ev}} 
\newcommand{\lspan}{{\rm span}} 
\newcommand{\abs}[1]{\left|#1\right|} 
\newcommand{\norm}[1]{\left\|#1\right\|} 
\newcommand{\defeq}{\mathrel{:=}} 

\newcommand{\dop}{\text{: }} 

\newcommand{\supp}{{\rm supp}\,}
\newcommand{\suppc}{{\rm supp}}
\newcommand{\dom}{{\rm dom}\,}
\newcommand{\ran}{{\rm ran}\,}

\newcommand{\lge}{\left\{} 
\newcommand{\rge}{\right\}} 
\newcommand{\lru}{\left(} 
\newcommand{\rru}{\right)} 
\newcommand{\leck}{\left[} 
\newcommand{\reck}{\right]} 
\newcommand{\rukl}[1]{\lru #1 \rru} 
\newcommand{\eckl}[1]{\leck #1 \reck} 
\newcommand{\gekl}[1]{\lge #1 \rge} 

\newcommand{\menge}[2]{\gekl{ #1 \dop #2 }} 
\newcommand{\Ext}{{\rm Ext}\,}
\title{Cartan subalgebras and the UCT problem}

\author{Sel\c{c}uk Barlak}
\address{Department of Mathematics and Computer Science\\
University of Southern Denmark\\
Campusvej 55\\
DK-5230 Odense M\\
Denmark}
\email{barlak@imada.sdu.dk}

\author{Xin Li}
\address{School of Mathematical Sciences\\
Queen Mary University of London\\
Mile End Road\\
London E1 4NS}
\email{xin.li@qmul.ac.uk}

\thanks{The first named author is supported by SFB 878 \emph{Groups, Geometry and Actions}, GIF Grant 1137-30.6/2011, ERC AdG 267079, and the Villum Fonden project grant `Local and global
structures of groups and their algebras' (2014–-2018).}
\thanks{The second named author is supported by EPSRC grant EP/M009718/1.}

\begin{document}

\begin{abstract}
We show that a separable, nuclear C*-algebra satisfies the UCT if it has a Cartan subalgebra. Furthermore, we prove that the UCT is closed under crossed products by group actions which respect Cartan subalgebras. This observation allows us to deduce, among other things, that a crossed product $\cO_2\rtimes_\alpha \Zz_p $ satisfies the UCT if there is some automorphism $\gamma$ of $\cO_2$ with the property that $\gamma(\cD_2)\subseteq \cO_2\rtimes_\alpha \Zz_p$ is regular, where $\cD_2$ denotes the canonical masa of $\cO_2$. We prove that this condition is automatic if $\gamma(\cD_2)\subseteq \cO_2\rtimes_\alpha \Zz_p$ is not a masa or $\alpha(\gamma(\cD_2))$ is inner conjugate to $\gamma(\cD_2)$. Finally, we relate the UCT problem for separable, nuclear, $M_{2^\infty}$-absorbing C*-algebras to Cartan subalgebras and order two automorphisms of $\cO_2$.
\end{abstract}

\maketitle



\section{Introduction}

As defined by Rosenberg and Schochet in \cite{RSch}, a separable C*-algebra $A$ is said to satisfy the universal coefficient theorem (UCT) if for every separable C*-algebra $A'$, the following sequence is exact
$$
 0 \to \Ext(K_*(A),K_{*-1}(A')) \to KK_*(A,A') \to\Hom(K_*(A),K_*(A')) \to 0,
$$
where the right hand map is the natural one and the left hand map is the inverse of a map that is always defined. They showed that a separable C*-algebra satisfies the UCT if and only if it is $KK$-equivalent to a commutative C*-algebra. The class of separable, nuclear C*-algebras satisfying the UCT is called the bootstrap category, denoted $\cN$, and can alternatively be characterized as the smallest class of separable, nuclear C*-algebras that contains $\Cz$ and is closed under countable inductive limits, the two out of three property for extensions, and $KK$-equivalences, see \cite{Bla}.

Many natural constructions of C*-algebras are known to preserve the bootstrap category. For example, it follows from work of Cuntz \cite{Cu} that $\cN$ is stable under taking crossed products by $\Zz$. Moreover, Fack and Skandalis's result that the Connes' Thom isomorphism is a $KK$-equivalence yields that the bootstrap category is also stable under taking crossed by $\Rz$, see \cite{FS}. In his remarkable paper \cite{Tu}, Tu has proven that all C*-algebras associated with Hausdorff, locally compact, second countable, amenable groupoids satisfy the UCT. This is a huge class of C*-algebras and includes for example all transformation group C*-algebras associated with countable groups acting amenably on locally compact, second countable spaces. Whereas Skandalis has given in \cite{S} an example of an exact, non-nuclear C*-algebra not satisfying the UCT, it remains a major open question whether all separable, nuclear C*-algebras satisfy the UCT. This question is often referred to as the UCT problem. 

Like for groups, a twisted groupoid $(G,\Sigma)$ of two topological groupoids $G$ and $\Sigma$ is a central groupoid extension
$$
\Tz \times G^{(0)} \rightarrowtail \Sigma \twoheadrightarrow G.
$$
Given a twisted groupoid $(G,\Sigma)$, and provided that $G$ is Hausdorff, locally compact, second countable and possesses a Haar system, one can form the reduced C*-algebra $C_r^*(G,\Sigma)$, see for example \cite{Ku1,Ku2,R}. Our first main result is the following:
\btheo \label{intro TwistedGPD-UCT}
Assume that $(G,\Sigma)$ is a twisted groupoid with $G$ {\'e}tale, Hausdorff, locally compact, second countable. If $C^*_r(G,\Sigma)$ is nuclear, then it satisfies the UCT.
\etheo
The proof of this theorem relies heavily on Tu's results and techniques. To make these applicable to our setup, we make use of recent results by Takeishi \cite{Ta} and by van Erp and Williams \cite{EW}.

Restricting to topologically principal groupoids, our result also covers the case of separable, nuclear C*-algebras admitting a Cartan subalgebra in the sense of Renault, see \cite{R}. Recall that a Cartan subalgebra $B$ of a C*-algebra $A$ is a regular masa that contains an approximate identity for $A$ and admits a faithful conditional expectation $A \onto B$, see Section \ref{Prelim}. Renault's remarkable result in \cite{R} states that up to isomorphism every Cartan pair $(A,B)$ arises as the reduced C*-algebra of a twisted groupoid $(G,\Sigma)$ with $G$ {\'e}tale, Hausdorff, locally compact, second countable and topologically principal, and with the Cartan subalgebra being the algebra of continuous functions vanishing at infinity on the unit space $G^{(0)}$. Conversely, every such pair $(C^*_r(G,\Sigma),C_0(G^{(0)}))$ is shown to be a Cartan pair. As a consequence, we obtain:

\bcor \label{intro Cartan-UCT}
Let $A$ be a separable and nuclear C*-algebra. If $A$ has a Cartan subalgebra, then $A$ satisfies the UCT.
\ecor

In combination with results by Katsura \cite{Kat} and Yeend \cite{Y1, Y2}, Corollary~\ref{intro Cartan-UCT} in particular implies that a Kirchberg algebra satisfies the UCT if and only if it admits a Cartan subalgebra, see Remark~\ref{UCT <-> Kirchberg Cartan}.

The aim of the second part of this paper is to provide a first step towards a better understanding of the connection between Cartan subalgebras and the UCT problem. Over the years, several reformulations of the UCT problem have been established. As outlined in \cite[23.15.12]{Bla}, the UCT problem has a positive answer if and only if $\cN$ is closed under crossed products by the circle group $\Tz$. A deep result of Kirchberg  states that every separable, nuclear C*-algebra is $KK$-equivalent to a Kirchberg algebra, see \cite[Theorem I]{Kir}. Building upon this result, one can show that the UCT problem reduces to the question whether every crossed product of $\cO_2$ by an outer $\Zz_p$-action, with $p$ prime, satisfies the UCT, see \cite[Theorem 4.17]{BS}. 

Having in mind the characterization of  the UCT problem in terms of finite cyclic group actions on $\cO_2$ and considering Theorem \ref{intro TwistedGPD-UCT}, it is natural to ask the following question:

\bquestion
Let $(A,B)$ be a Cartan pair with $A$ separable and nuclear, and let $\alpha \in \Aut(A)$ be of finite order, say $n$. When does $A\rtimes_\alpha \Zz_n$ satisfy the UCT?
\equestion

This question is addressed in the second part of the paper and some partial answer will be given. As the formulation of our main result in this direction would be rather technical, we restrict at this point to reduced C*-algebras associated with minimal and purely infinite groupoids in the sense of Matui, see \cite{Ma}.

\btheo \label{intro UCTregular}
Let $G$ be an {\'e}tale, Hausdorff, locally compact, second countable, topologically principal groupoid that is minimal, purely infinite and amenable, and that has the property that $G^{(0)}$ is homeomorphic to the Cantor set. Let $p$ be a prime number and $\alpha:\Zz_p\curvearrowright C_r^*(G)$ an action. Suppose that there exists an automorphism $\gamma\in \Aut(C_r^*(G))$ such that $
\gamma(C(G^{(0)}))$ is a regular sub-$C^*$-algebra of $C_r^*(G)\rtimes_\alpha \Zz_p$. Then $C_r^*(G)\rtimes_\alpha \Zz_p$ satisfies the UCT.
\etheo

Theorem \ref{intro UCTregular} also applies to the Cartan pair $(\cO_2,\cD_2)$, where $\cD_2$ denotes the canonical masa of $\cO_2$, and therefore relates to the UCT problem. If $\gamma(C(G^{(0)})) \subseteq C_r^*(G)\rtimes_\alpha \Zz_p$ is a masa, then the claim follows from Corollary \ref{intro Cartan-UCT}. In order to handle the case that $\gamma(C(G^{(0)})) \subseteq C_r^*(G) \rtimes_\alpha \Zz_p$ is not a masa, we make use of the following two results, which should be of independent interest. The first one holds in greater generality than stated here.

\btheo \label{intro characterizationDiagonalNotMasa}
Let $G$ be an {\'e}tale, Hausdorff, locally compact, second countable, topologically principal groupoid that is minimal and purely infinite, and that has the property that $G^{(0)}$ is homeomorphic to the Cantor set. Let $p$ be a prime number and $\alpha:\Zz_p\curvearrowright A$ an action. Then the following are equivalent:
\begin{enumerate}
 \item[(i)]
  $C(G^{(0)}) \subseteq C^*_r(G)\rtimes_\alpha \Zz_p$ is not a masa.
  \item[(ii)]
   $\alpha$ is exterior equivalent to a $\Zz_p$-action fixing $C(G^{(0)})$ pointwise.
\end{enumerate}
\etheo

\bprop \label{intro Gamma.B=B}
Let $A$ be a separable C*-algebra and assume that $B \subseteq A$ is a Cartan subalgebra. Let $\Gamma$ be a countable group acting on $A$ such that $B$ is globally invariant under this action.

If $A \rtimes_r \Gamma$ is nuclear, then $A \rtimes_r \Gamma$ satisfies the UCT.
\eprop

The proof of the second result again uses Renault's characterization of $A$ as a reduced twisted groupoid C*-algebra $C^*_r(G,\Sigma)$. The $\Gamma$-action on $A$ gives rise to actions on $G$ and $\Sigma$, and we identify the reduced crossed product $A \rtimes_r \Gamma$ as the reduced C*-algebra of the twisted semidirect product groupoid $(G \ltimes \Gamma,\Sigma \ltimes \Gamma)$. Proposition \ref{intro Gamma.B=B} then follows from Theorem \ref{intro TwistedGPD-UCT}.

In the special case of outer strongly approximately inner $\Zz_2$-actions on $\cO_2$ in the sense of \cite[Definition~3.6]{Iz}, we obtain an equivalent characterization for the UCT of the associated crossed products:
\btheo \label{intro char str appr inner UCT}
Let $\beta: \: \Zz_2 \curvearrowright \cO_2$ be outer strongly approximately inner. $\cO_2 \rtimes_{\beta} \Zz_2$ satisfies the UCT if and only if there exists a Cartan subalgebra in $\cO_2$ which is fixed by $\beta$ pointwise.
\etheo

Combining Proposition~\ref{intro Gamma.B=B} and Theorem~\ref{intro char str appr inner UCT} with Izumi's classification of outer strongly approximately inner $\Zz_2$-actions on $\cO_2$ (see \cite[Theorem~4.8]{Iz}), we provide the following new characterization of the UCT problem for separable, nuclear C*-algebras that are $KK$-equivalent to their $M_{2^\infty}$-stabilization:

\btheo \label{char UCT problem}
The following statements are equivalent:
\begin{enumerate}
\item[(i)] Every separable, nuclear C*-algebra $A$ that is $KK$-equivalent to $A \otimes M_{2^\infty}$ satisfies the UCT.
\item[(ii)] For every unital, $M_{2^\infty}$-absorbing Kirchberg algebra $A$ in Cuntz standard form, there exists a strongly approximately inner action $\beta:\Zz_2 \curvearrowright \cO_2$ and a Cartan subalgebra $B \subseteq \cO_2$ such that $\cO_2 \rtimes_\beta \Zz_2 \cong A$ and $\beta$ fixes $B$ pointwise.
\item[(iii)] Every outer strongly approximately inner $\Zz_2$-action on $\cO_2$ fixes some Cartan subalgebra $B \subseteq \cO_2$ pointwise.
\item[(iv)] Every outer strongly approximately inner $\Zz_2$-action on $\cO_2$ fixes some Cartan subalgebra $B \subseteq \cO_2$ globally.
\end{enumerate}
\etheo

The paper is organized as follows. In Section~\ref{Prelim}, we mainly recall Renault's construction of a twisted groupoid associated with a Cartan pair. In Section~\ref{The UCT for nuclear C*-algebras with a Cartan subalgebra}, we prove Theorem~\ref{intro TwistedGPD-UCT}, Corollary~\ref{intro Cartan-UCT} and Proposition~\ref{intro Gamma.B=B}. In Section~\ref{Masas in Crossed Products}, we first make some observations about the interplay of masas and automorphisms of C*-algebras with an emphasis on masas with totally disconnected spectrum. We then prove (generalizations of) Theorem~\ref{intro characterizationDiagonalNotMasa} and \ref{intro UCTregular}. In Section~\ref{Section UCT}, we prove Theorem~\ref{intro char str appr inner UCT} and Theorem~\ref{char UCT problem}.

\section{Preliminaries}
\label{Prelim}

\bdefin[{\cite[Definition~5.1]{R}}]
\label{CartanSubalgebra}
A sub-C*-algebra $B$ of a C*-algebra $A$ is called a Cartan subalgebra if
\begin{enumerate}
\item[(i)] $B$ contains an approximate identity of $A$;
\item[(ii)] $B$ is maximal abelian;
\item[(iii)] $B$ is regular, i.e., $N_A(B) \defeq \menge{n \in A}{n B n^* \subseteq B \ \text{and} \ n^* B n \subseteq B}$ generates $A$ as a C*-algebra;
\item[(iv)] there exists a faithful conditional expectation $P:A \onto B$.
\end{enumerate}

A pair $(A,B)$, where $B$ is a Cartan subalgebra of a C*-algebra $A$, is called a Cartan pair.
\edefin

Renault shows that every Cartan pair $(A,B)$, with $A$ separable, arises -- up to isomorphism -- from a twisted {\'e}tale Hausdorff locally compact second countable topologically principal groupoid. Let us present a few more details: 
\bdefin
A twisted groupoid $(G,\Sigma)$ consists of two topological groupoids $G$ and $\Sigma$, together with a central groupoid extension
$$
\Tz \times G^{(0)} \rightarrowtail \Sigma \twoheadrightarrow G
$$
where $\Tz$ is the circle group.
\edefin
In the following, we always denote by $r$ and $s$ the range and source map of a given groupoid.

\bremark
For later purposes, let us briefly explain the construction of the reduced C*-algebra $C^*_r(G,\Sigma)$ of a twisted groupoid $(G,\Sigma)$, under the assumption that $G$ is locally compact, Hausdorff, second countable and possesses a Haar system $\menge{\lambda_x}{x\in G^{(0)}}$. Details can be found at the beginning of \cite[\S~4]{R}. On
$$
C_c(G,\Sigma) = \menge{f\in C_c(\Sigma,\Cz)}{f(z\sigma)=f(\sigma)\bar{z} \ {\rm for} \ {\rm all} \ \sigma \in \Sigma,\ z \in \Tz},
$$
define convolution and involution via the formulas
$$
f * g (\sigma) = \int f(\sigma\tau^{-1})g(\tau)\mathrm d\lambda_{s(\sigma)}(\dot{\tau}) \ \text{and}\ f^*(\sigma) = \overline{f(\sigma^{-1})}.
$$
Here $\dot{\tau}\in G$ denotes the image of $\tau\in \Sigma$ under the surjection $\Sigma \onto G$. With these operations, $C_c(G,\Sigma)$ turns into a $*$-algebra. For $x\in G^{(0)}$, consider the Hilbert space
$$
\mathscr{H}_x = \menge{\xi: \: \Sigma_x \to \Cz}{\xi(z\sigma) = \xi(\sigma)\bar{z}, \, [\dot{\sigma} \mapsto |\xi(\dot{\sigma})|^2] \text{ measurable}, \, \int \abs{\xi(\dot{\sigma})}^2 \mathrm d \lambda_x(\dot{\sigma}) < \infty},
$$
and let $\pi_x$ denote the $*$-representation of $C_c(G,\Sigma)$ on $\mathscr{H}_x$ given by
$$
\pi_x(f)\xi(\sigma) = \int f(\sigma\tau^{-1})\xi(\tau) \mathrm d\lambda_x(\dot{\tau}).
$$
Now, $C^*_r(G,\Sigma)$ is the completion of $C_c(G,\Sigma)$ with respect to the norm $\|f\| = \sup\limits_{x \in G^{(0)}} \|\pi_x(f)\|$.
\eremark

\bdefin
A topological groupoid $G$ is {\'e}tale if the range and source maps are local homeomorphisms from $G$ onto $G^{(0)}$.
\edefin
If $G$ is {\'e}tale, then there always exists a Haar system (counting measures) on $G$, and there is a canonical embedding $C_0(G^{(0)}) \into C^*_r(G,\Sigma)$.
\bdefin
A topological groupoid $G$ is topologically principal if the set of points of $G^{(0)}$ with trivial isotropy is dense in $G^{(0)}$, where $x \in G^{(0)}$ is said to have trivial isotropy if $\menge{\gamma \in G}{r(\gamma) = s(\gamma) = x} = \gekl{x}$.
\edefin

\bdefin
A twisted groupoid $(G,\Sigma)$ is called a twisted {\'e}tale Hausdorff locally compact second countable topologically principal groupoid if $G$ is {\'e}tale, Hausdorff, locally compact, second countable and topologically principal.
\edefin

With these notations, we are now ready to state Renault's theorem:
\btheo[{\cite[Theorem~5.2 and Theorem~5.9]{R}}]
\label{Renault}
Cartan pairs $(A,B)$, where $A$ is a separable C*-algebra, are precisely of the form $(C^*_r(G,\Sigma),C_0(G^{(0)}))$, where $(G,\Sigma)$ is a twisted {\'e}tale Hausdorff locally compact second countable topologically principal groupoid.
\etheo

\bremark
\label{A,B-->GPD}
For later purposes, let us briefly explain how $(G,\Sigma)$ is constructed out of $(A,B)$. Set $X \defeq \Spec(B)$. By \cite[Proposition~4.7]{R}, for every $n \in N_A(B)$ there exists a partial homeomorphism $\alpha_n: \: \dom(n) \to \ran(n)$, where $\dom(n) = \menge{x \in X}{n^*n(x) > 0}$ and $\ran(n) = \menge{x \in X}{nn^*(x) > 0}$, such that $n^*bn(x) = b(\alpha_n(x))n^*n(x)$ for all $b \in B$ and $x \in \dom(n)$. Here we use the canonical identification $B \cong C_0(X)$. Define the pseudogroup $\cG(B) \defeq \menge{\alpha_n}{n \in N_A(B)}$ on $X$. Let $G(B)$ be the groupoid of germs of $\cG(B)$, i.e., $G(B) = \menge{[x,\alpha_n,y]}{n \in N_A(B), \, y \in \dom(n), \, x = \alpha_n(y)}$. Here $[x,\alpha_n,y] = [x,\alpha_{n'},y]$ if there exists an open neighbourhood $V$ of $y$ in $X$ such that $\alpha_n \vert_V = \alpha_{n'} \vert_V$. To describe the twist, set $D \defeq \menge{(x,n,y) \in X \times N_A(B) \times X}{y \in \dom(n), \, x = \alpha_n(y)}$ and define $\Sigma(B) \defeq D / {}_{\sim}$, where $(x,n,y) \sim (x',n',y')$ if $y=y'$ and there exist $b,b' \in B$ with $b(y),b'(y) > 0$ and $nb=n'b'$. $G(B)$ and $\Sigma(B)$ are topological groupoids (see \cite{R} for details), and the canonical homomorphism $\Sigma(B) \to G(B), \, [x,n,y] \ma [x,\alpha_n,y]$, yields the central extension
$$
\Tz \times X \rightarrowtail \Sigma(B) \twoheadrightarrow G(B).
$$

We can now identify $(A,B)$ with $(C_r^*(G(B),\Sigma(B)),C_0(X))$ as follows. For $a\in N_A(B)$, the function $\hat{a}:D \to \Cz$, $\hat{a}(x,n,y) = P(n^*a)(y)/\sqrt{n^*n(y)}$, gives rise to a well-defined, continuous map $\hat{a}:\Sigma(B)\to \Cz$. If $\hat{a}$ has compact support, then $\hat{a}\in C_c(G(B),\Sigma(B))$. It is shown in \cite[\S~5]{R} that the linear map $\lspan(\menge{a \in N_A(B)}{\suppc(\hat{a}) \ \text{compact}}) \to C_c(G(B),\Sigma(B))$, $a\mapsto \hat{a}$ is an isometric $*$-isomorphism and hence extends uniquely to a Cartan isomorphism $(A,B) \cong (C^*_r(G(B),\Sigma(B)),C_0(X))$.
\eremark

\section{The UCT for nuclear C*-algebras with a Cartan subalgebra}
\label{The UCT for nuclear C*-algebras with a Cartan subalgebra}

In this section, we show that a separable, nuclear C*-algebra satisfies the UCT if it is isomorphic to the reduced C*-algebra of a twisted {\'e}tale Hausdorff locally compact second countable groupoid. In particular, applying Renault's characterization of Cartan pairs, we conclude that a separable, nuclear C*-algebra satisfies the UCT if it admits a Cartan subalgebra. Furthermore, we show that a nuclear C*-algebra satisfies the UCT if it can be written as a crossed product C*-algebra $A\rtimes_{\alpha,r} \Gamma$, where $A$ is a separable C*-algebra admitting a Cartan subalgebra $B$, $\Gamma$ a countable, discrete group, and $\alpha$ an action of $\Gamma$ on $A$ leaving $B$ globally invariant.

For the definition of groupoid actions on C*-algebras and groupoid crossed products, we refer the reader to \cite[\S~10]{MW}. For a definition of proper groupoid actions on topological spaces and C*-algebras, the reader may consult \cite[\S~1]{Tu2} and \cite[\S~1]{Tu}.

\btheo
\label{TwistedGPD-UCT}
Assume that $(G,\Sigma)$ is a twisted {\'e}tale Hausdorff locally compact second countable groupoid. If $C^*_r(G,\Sigma)$ is nuclear, then it satisfies the UCT.
\etheo
\bproof
We start with the following result (see \cite[Proposition~5.1]{EW}): There is a Hilbert $C_0(G^{(0)})$-module ${\rm \bf H}$ and a $G$-action on ${\bf A} = \cK( {\rm \bf H} )$ such that ${\bf A} \rtimes_r G$ is Morita equivalent to $C^*_r(G,\Sigma)$.
In particular, ${\bf A}$ is a $C_0(G^{(0)})$-algebra, with fibre ${\bf A}_x = \cK( {\rm \bf H} (x))$. Here, ${\rm \bf H}(x)$ is the Hilbert space arising as the pushout of ${\rm \bf H}$ by the evaluation map at $x$, $\ev_x:C_0(G^{(0)}) \to \Cz$. Indeed, ${\bf A}$ is a $C_0(G^{(0)})$-algebra in this way as there is a topology on $\mathscr{H} = \bigsqcup_{x \in G^{(0)}} {\rm \bf H}(x)$ turning $\mathscr{H} \onto G^{(0)}$ into a continuous Hilbert bundle such that ${\rm \bf H}$ is isomorphic to the space of continuous sections vanishing at infinity $\Gamma_0(G^{(0)}; \mathscr{H})$, see \cite[Theorem~II.13.18]{FD}.

If $C^*_r(G,\Sigma)$ is nuclear, then by \cite[Theorem~5.4]{Ta}, $G$ must be amenable.

Since $G$ is amenable, it acts properly on a continuous field $H$ of affine Euclidean spaces by \cite[Lemme~3.5]{Tu}. Hence, by \cite[\S~9]{Tu}, there exists a $G$-action on a C*-algebra $\cA(H)$ and $\eta \in KK_G(C_0(G^{(0)}),\cA(H))$, $D \in KK_G(\cA(H),C_0(G^{(0)}))$ such that $\eta \otimes_{\cA(H)} D = {\bf 1}_{C_0(G^{(0)})}$. An explicit construction of the C*-algebra $\cA(H)$ appears in \cite[\S~7]{Tu}. Moreover, by \cite[Lemme~7.2]{Tu} the $G$-action on $\cA(H)$ is proper. This means that there is a proper $G$-space $Z$ (constructed in \cite[\S~6]{Tu}) such that $\cA(H)$ is a $Z \rtimes G$-algebra. Hence, $\cA(H) \otimes_{C_0(G^{(0)})} {\bf A}$ is a $Z \rtimes G$-algebra as well.

Using the canonical homomorphisms
$$
KK_G(\cA(H),C_0(G^{(0)})) \to KK_G(\cA(H) \otimes_{C_0(G^{(0)})} {\bf A}, C_0(G^{(0)}) \otimes_{C_0(G^{(0)})} {\bf A})
$$
and 
$$
KK_G(C_0(G^{(0)}),\cA(H)) \to KK_G(C_0(G^{(0)}) \otimes_{C_0(G^{(0)})} {\bf A}, \cA(H) \otimes_{C_0(G^{(0)})} {\bf A})
$$ 
(see \cite[Definition~6.2]{LeG}) which we denote by $\tau_{\bf A}$, we obtain $\tau_{\bf A}(\eta) \otimes_{\cA(H) \otimes_{C_0(G^{(0)})} {\bf A}} \tau_{\bf A}(D) = {\bf 1}_{\bf A}$ in $KK_G({\bf A},{\bf A})$. Now set ${\bf y} \defeq j_G(\tau_{\bf A}(\eta)) \in KK({\bf A} \rtimes_r G, (\cA(H) \otimes_{C_0(G^{(0)})} {\bf A}) \rtimes_r G)$ and ${\bf x} \defeq j_G(\tau_{\bf A}(D)) \in KK((\cA(H) \otimes_{C_0(G^{(0)})} {\bf A}) \rtimes_r G, {\bf A} \rtimes_r G)$, where $j_G$ is the descent homomorphism, see \cite[\S 7.2]{LeGTh}. Then ${\bf y} {\bf x} = {\bf 1}_{{\bf A} \rtimes_r G}$ in $KK({\bf A} \rtimes_r G,{\bf A} \rtimes_r G)$. Using that ${\bf A} \rtimes_r G$ is Morita equivalent to $C_r^*(G,\Sigma)$, \cite[Corollary~23.10.8]{Bla} now yields that it suffices to show that $(\cA(H) \otimes_{C_0(G^{(0)})} {\bf A}) \rtimes_r G$ satisfies the UCT.

We now follow the argument in \cite[\S~10]{Tu}. It is shown in the proof of \cite[Lemme~10.6]{Tu} that there are $Z \rtimes G$-subalgebras $\cA(H)_{\ve}$, $\ve > 0$, of $\cA(H)$ such that $\cA(H) = \overline{\bigcup_{\ve > 0} \cA(H)_{\ve}}$. In addition, for every $x \in G^{(0)}$, the fibre $(\cA(H)_{\ve})_x$ is type I. Therefore, $\cA(H)_{\ve} \otimes_{C_0(G^{(0)})} {\bf A}$, $\ve > 0$, are $Z \rtimes G$-subalgebras of $\cA(H) \otimes_{C_0(G^{(0)})} {\bf A}$ such that $\cA(H) \otimes_{C_0(G^{(0)})} {\bf A} = \overline{\bigcup_{\ve > 0} \cA(H)_{\ve} \otimes_{C_0(G^{(0)})} {\bf A}}$. For every $x \in G^{(0)}$, the corresponding fibre of $\cA(H)_{\ve} \otimes_{C_0(G^{(0)})} {\bf A}$ is given by $(\cA(H)_{\ve})_x \otimes {\bf A}_x$, and since $(\cA(H)_{\ve})_x$ is type I and ${\bf A}_x = \cK({\bf H}(x))$, this means that every fibre $\rukl{\cA(H)_{\ve} \otimes_{C_0(G^{(0)})} {\bf A}}_x$ is type I. As $Z \rtimes G$ is a proper groupoid (see \cite[\S~1]{Tu}), we can apply \cite[Proposition~10.3]{Tu} and obtain that 
$$
\rukl{\cA(H)_{\ve} \otimes_{C_0(G^{(0)})} {\bf A}} \rtimes_r G \cong \rukl{\rukl{\cA(H)_{\ve} \otimes_{C_0(G^{(0)})} {\bf A}} \rtimes_r Z} \rtimes_r G \cong \rukl{\cA(H)_{\ve} \otimes_{C_0(G^{(0)})} {\bf A}} \rtimes_r (Z \rtimes G)
$$
is type I. Therefore, $(\cA(H) \otimes_{C_0(G^{(0)})} {\bf A}) \rtimes_r G = \overline{ \bigcup_{\ve > 0} \rukl{\cA(H)_{\ve} \otimes_{C_0(G^{(0)})} {\bf A}} \rtimes_r G }$ satisfies the UCT, being an inductive limit of type I C*-algebras.
\eproof

Note that in Theorem~\ref{TwistedGPD-UCT}, we did not need to assume that $G$ is topologically principal. However, restricted to that case, Theorem~\ref{TwistedGPD-UCT} gives in combination with Theorem~\ref{Renault} the following
\bcor \label{Cartan-UCT}
Let $A$ be a separable and nuclear C*-algebra. If $A$ has a Cartan subalgebra, then $A$ satisfies the UCT.
\ecor

\bremark \label{UCT <-> Kirchberg Cartan}
Every UCT Kirchberg algebra has a Cartan subalgebra. That follows from \cite[Thereom~C]{Kat} and \cite{Y1,Y2}, see also \cite{RSWY}. As every separable, nuclear C*-algebra is $KK$-equivalent to a Kirchberg algebra by \cite[Theorem I]{Kir}, we therefore conclude that the UCT problem has a positive answer if and only if every Kirchberg algebra admits a Cartan subalgebra.
\eremark

\bprop
\label{Gamma.B=B}
Let $A$ be a separable C*-algebra admitting a Cartan subalgebra $B \subseteq A$. Let $\Gamma$ be a countable group acting on $A$ via $\Gamma \times A \to A, \, (\gamma,a) \ma \gamma.a$. Assume that for every $\gamma \in \Gamma$, $\gamma.B = B$.

If $A \rtimes_r \Gamma$ is nuclear, then $A \rtimes_r \Gamma$ satisfies the UCT.
\eprop
\bproof
By assumption $\Gamma \curvearrowright A$ restricts to a $\Gamma$-action $\Gamma \curvearrowright B$. Dualizing, we obtain a $\Gamma$-action $\Gamma \curvearrowright X$ denoted by $\Gamma \times X \to X, \, (\gamma,x) \ma \gamma.x$, such that $(\gamma.b)(x) = b(\gamma^{-1}.x)$ for all $b \in B$, $x \in X$. Here we use the canonical identification $B \cong C_0(X)$. By Theorem~\ref{Renault}, we know that $(A,B) \cong (C^*_r(G,\Sigma),C_0(X))$ for a uniquely determined twist $(G,\Sigma)$. Looking at the construction of $(G,\Sigma)$ in Remark~\ref{A,B-->GPD}, it is obvious that we have $\Gamma$-actions $\Gamma \curvearrowright G$, $\Gamma \curvearrowright \Sigma$ given by $\gamma.[x,\alpha_n,y] = [\gamma.x,\alpha_{\gamma.n},\gamma.y]$ and $\gamma.[x,n,y] = [\gamma.x,\gamma.n,\gamma.y]$.

Moreover, under the identification $(A,B) \cong (C^*_r(G,\Sigma),C_0(X))$, $\gamma.f(\sigma) = f(\gamma^{-1}.\sigma)$ for all $\gamma\in \Gamma$, $\sigma\in \Sigma$, and $f\in C_c(G,\Sigma)$. Indeed, this follows from Remark~\ref{A,B-->GPD}, as if $a\in N_A(B)$ and $\hat{a}$ has compact support, then
\bglnoz
  && \hat{a}([\gamma^{-1}.x,\gamma^{-1}.n,\gamma^{-1}.y])
  = P((\gamma^{-1}.n)^*a)(\gamma^{-1}.y)/\sqrt{\gamma^{-1}.(n^*n)(\gamma^{-1}.y)} \\
  &=& (\gamma.P(\gamma^{-1}.(n^*\gamma.a)))(y) / \sqrt{n^*n(y)}
  = \widehat{\gamma.a}([x,n,y]).
\eglnoz
The last equality follows from the fact that $P$ is the unique conditional expectation from $A$ onto $B$, see \cite[Corollary 5.10]{R}.

Clearly, the canonical projection $\Sigma \to G$ is $\Gamma$-equivariant. We now form semidirect products: Let $\Gamma \ltimes G$ be the groupoid with underlying set $\Gamma \times G$, multiplication $(\gamma,g)(\gamma',g') = (\gamma \gamma', ({\gamma'}^{-1}.g)g')$, and inversion $(\gamma,g)^{-1} = (\gamma^{-1},\gamma.g^{-1})$. Note that $s((\gamma,g)) = (e,s(g))$ and $r((\gamma,g)) = (e,\gamma.r(g))$, which implies that $(\Gamma \ltimes G)^{(0)} = \gekl{e} \times  G^{(0)} \cong X$. Similarly, let $\Gamma \ltimes \Sigma$ be the groupoid with underlying set $\Gamma \times \Sigma$, multiplication $(\gamma,\sigma)(\gamma',\sigma') = (\gamma \gamma', ({\gamma'}^{-1}.\sigma)\sigma')$, and inversion $(\gamma,\sigma)^{-1} = (\gamma^{-1},\gamma.\sigma^{-1})$. Both, $\Gamma \ltimes G$ and $\Gamma \ltimes \Sigma$ are topological groupoids when equipped with the respective product topologies. As before, the homomorphism $\Gamma \ltimes \Sigma \to \Gamma \ltimes G, \, (\gamma,\sigma) \to (\gamma,\dot{\sigma})$, yields an extension
$$
\Tz \times X \rightarrowtail \Gamma \ltimes \Sigma \twoheadrightarrow \Gamma \ltimes G.
$$
One checks that $(\Gamma \ltimes G,\Gamma \ltimes \Sigma)$ is indeed a twisted {\'e}tale Hausdorff locally compact second countable groupoid. Observe also that
$$
C_c(\Gamma \ltimes G, \Gamma\ltimes \Sigma) = \menge{\phi \in C_c(\Gamma \ltimes \Sigma,\Cz)}{\phi(\gamma,z\sigma) = \phi(\gamma,\sigma) \bar{z} \ {\rm for} \ {\rm all} \ z \in \Tz}.
$$

Our goal is to show that $C^*_r(\Gamma \ltimes G,\Gamma \ltimes \Sigma) \cong A \rtimes_r \Gamma$. By Theorem~\ref{TwistedGPD-UCT}, it then follows that $A \rtimes_r \Gamma$ satisfies the UCT. More precisely, we show that the linear map $C_c(\Gamma \ltimes G, \Gamma\ltimes \Sigma) \to C_c(\Gamma,A)$
determined by $\delta_{\gamma} \otimes f \ma \delta_{\gamma} * f$, where $f \in C_c(G,\Sigma)$, is an isometric $*$-isomorphism and hence extends to an isomorphism $C^*_r(\Gamma \ltimes G,\Gamma \ltimes \Sigma) \cong A \rtimes_r \Gamma$.

For $x\in X$, let $\pi_x$ denote the representation of $C_c(G,\Sigma)$ introduced in Section \ref{Prelim}. We denote by $(\Gamma \ltimes \pi)_x$ the analogous representation associated with the twisted groupoid $(\Gamma \ltimes G, \Gamma \ltimes \Sigma)$ on the Hilbert space $\mathscr{K}_x$. Moreover, let $\tilde{\pi}_x$ denote the unique representation of $A$ on $\mathscr{H}_x$ extending $\pi_x$. By definition,
$$
\norm{\delta_{\gamma} \otimes f}_{C^*_r(\Gamma \ltimes G,\Gamma \ltimes \Sigma)} = \sup_{x \in X} \norm{(\Gamma \ltimes \pi)_x(\delta_{\gamma} \otimes f)}.
$$
Moreover,
$$
\norm{\delta_{\gamma} * f}_{A \rtimes_r \Gamma} = \sup_{x \in X} \norm{(\ti{\pi}_x \rtimes \Gamma)(\delta_{\gamma} * f)},
$$
where $\tilde{\pi}_x \rtimes \Gamma$ denotes the integrated form associated with $\tilde{\pi}_x$. To see this, note that $\menge{\pi_x}{x\in X}$ is a continuous field of representation when $\mathscr{H} = \bigsqcup_{x \in X} \mathscr{H}_x$ carries the unique continuous Hilbert bundle structure such that $C_c(G,\Sigma)$ constitutes a fundamental family of continuous sections. The induced representation $\tilde{\pi}$ of $A$ on the Hilbert $C_0(X)$-module $\Gamma_0(X;\mathscr{H})$ is faithful, and the norm on $A\rtimes_r \Gamma$ is therefore induced from the representation $\tilde{\pi}\rtimes \Gamma$. However, $\tilde{\pi}\rtimes \Gamma$ itself is induced from the continuous field $\menge{\tilde{\pi}_x\rtimes \Gamma}{x\in X}$.

Let $U:\mathscr{H}_x \otimes \ell^2 \Gamma \to \mathscr{K}_x$ be the unitary determined by $\xi \otimes \delta_i \ma \delta_i \otimes \xi$. We will show that 
$$
U^* (\Gamma \ltimes \pi)_x (\delta_{\gamma} \otimes f) U = (\ti{\pi}_x \rtimes \Gamma)(\delta_{\gamma} * f),$$
which then yields the desired isomorphism $C^*_r(\Gamma \ltimes G,\Gamma \ltimes \Sigma) \cong A \rtimes_r \Gamma$. We compute
\bglnoz
  && (\Gamma \ltimes \pi)_x(\delta_{\gamma} \otimes f) U (\xi \otimes \delta_i)(j,\sigma)
  = (\Gamma \ltimes \pi)_x(\delta_{\gamma} \otimes f) (\delta_i \otimes \xi)(j,\sigma) \\
  &=& \sum_{(k,\dot{\tau}) \in (\Gamma \ltimes G)_x} (\delta_{\gamma} \otimes f)((j,\sigma)(k,\tau)^{-1})(\delta_i \otimes \xi)(k,\tau)
  = \sum_{(k,\dot{\tau})} (\delta_{\gamma} \otimes f)((jk^{-1},k.(\sigma \tau^{-1})) \delta_{i,k} \xi(\tau) \\
  &=& \sum_{(k,\dot{\tau})} \delta_{\gamma,jk^{-1}} \delta_{i,k} f(k.(\sigma \tau^{-1})) \xi(\tau)
  = \delta_{\gamma i} \otimes \pi_x(i^{-1}.f)(\xi)(j,\sigma).
\eglnoz
Therefore, $(\Gamma \ltimes \pi)_x(\delta_{\gamma} \otimes f) U (\xi \otimes \delta_i) = \delta_{\gamma i} \otimes \pi_x(i^{-1}.f)(\xi)$. By definition of the integrated form, we also have
\bglnoz
  && U ((\ti{\pi}_x \rtimes \Gamma)(\delta_{\gamma} * f)(\xi \otimes \delta_i))
  = U ((1 \otimes \lambda_{\gamma}) \pi_x(i^{-1}.f)(\xi) \otimes \delta_i) 
  = U (\pi_x(i^{-1}.f)(\xi) \otimes \delta_{\gamma i}) \\
  & = & \delta_{\gamma i} \otimes \pi_x(i^{-1}.f)(\xi).
\eglnoz
Thus, $U^* (\Gamma \ltimes \pi)_x (\delta_{\gamma} \otimes f) U = (\ti{\pi}_x \rtimes \Gamma)(\delta_{\gamma} * f)$, and the proof is complete.
\eproof

\bremark
Let $(A,B)$ be a Cartan pair with $A$ separable and $(G,\Sigma)$ the unique twist as in Theorem~\ref{Renault} such that $(A,B) \cong (C_r^*(G, \Sigma),C_0(G^{(0)}))$. The proof of Proposition~\ref{Gamma.B=B} also shows that any automorphism $\alpha \in \Aut(A)$ satisfying $\alpha(B) = B$ is induced by an isomorphism of the twist $(G,\Sigma)$.
\eremark

\section{Crossed products by finite cyclic groups}

\label{Masas in Crossed Products}

Let $A$ be a $C^*$-algebra, $B\subseteq A$ a masa, and $\alpha\in \Aut(A)$ an automorphism. Define
$$
F_{B,\alpha}:=\menge{x\in A}{xb=\alpha(b)x\quad \text{for all}\ b\in B}.
$$
This is a closed sub-vectorspace of $A$ with $\alpha(B)F_{B,\alpha}B \subseteq F_{B,\alpha}$ (with equality, if $B$ contains an approximate unit for $A$). In fact, for $x\in F_{B,\alpha}$ and $d,e,f\in B$,
$$
\alpha(e)xfd=\alpha(e)\alpha(d)xf=\alpha(d)\alpha(e)xf.
$$
Furthermore, $F_{B,\alpha}^*=F_{\alpha(B),\alpha^{-1}}$, since
$xb=\alpha(b)x$ if and only if $x^*\alpha(b)^*=b^*x=\alpha^{-1}(\alpha(b)^*)x^*$. Also note that for $x\in F_{B,\alpha}$ and $b\in B$,
$$
x^*xb=x^*\alpha(b)x=bx^*x\quad\text{and}\quad xx^*\alpha(b)=\alpha(b)xx^*,
$$
showing that $x^*x\in B$ and $xx^*\in \alpha(B)$. If $\beta\in \Aut(A)$ is another automorphism, then $\beta(F_{B,\alpha})=F_{\beta(B),\beta\alpha\beta^{-1}}$. In particular, $\alpha(F_{B,\alpha})=F_{\alpha(B),\alpha}$.

Assume now that $\alpha^n = \id$. The definition of $F_{B,\alpha}$ is exactly made for the purpose of characterizing when a masa $B \subseteq A$ is also a masa in $A \rtimes_\alpha \Zz_n$. The following statement holds in greater generality, but as we are only concerned with finite order automorphisms, we restrict to this case. In the following, let $t$ be the unitary in (the multiplier algebra of) $A \rtimes_{\alpha} \Zz_n$ implementing the $\Zz_n$-action $\alpha$.

\begin{proposition}
\label{characterizationMasa}
Let $A$ be a $C^*$-algebra, $B \subseteq A$ a masa, and $\alpha:\Zz_n\curvearrowright A$ an action. Then $B \subseteq A \rtimes_\alpha \Zz_n$ is a masa if and only if $F_{B,\alpha^k}= 0$ for $k=1,\ldots,n-1$.
\end{proposition}
\begin{proof}
Let $a=\sum_{k=0}^{n-1} t^kx_k\in A\rtimes_\alpha \Zz_n$. Then
$$
0=[a,b]=\sum\limits_{k=0}^{n-1} [t^kx_k,b]=\sum\limits_{k=0}^{n-1} t^k(x_kb-\alpha^{-k}(b)x_k)
$$
for all $b \in B$ if and only if $x_k\in F_{B,\alpha^{-k}}$ for $k=0,\ldots,n-1$. Now $B \subseteq A \rtimes_\alpha \Zz_n$ is a masa if and only if $A \rtimes_\alpha \Zz_n \cap B' = B$. This is exactly the case if $F_{B,\alpha^k}= 0$ for $k=1,\ldots,n-1$.
\end{proof}

One instance where $B \subseteq A \rtimes_\alpha \Zz_n$ is not a masa is when $\alpha$ restricts to the identity on $B$. In this section, we will see that for a certain class of masas with totally disconnected spectrum the converse is also true, at least when passing to an appropriate cocycle perturbation of $\alpha$. We then use this to show the UCT for the crossed products associated with such actions. This needs some preparation first.

\begin{lemma}
\label{automaticAlphaNormality}
Let $A$ be a $C^*$-algebra, $B\subseteq A$ a masa, and $\alpha\in \Aut(A)$. Then every $y\in F_{B,\alpha}$ satisfies $yy^*=\alpha(y^*y)$.
\end{lemma}
\begin{proof}
We have seen that $yy^*$, $\alpha(y^*y)\in \alpha(B)$. Set $X:=\Spec(\alpha(B))$. Assume that $yy^*\neq \alpha(y^*y)$. We may assume that there exists a non-empty open subset $U\subseteq X$ with the property that $yy^*(u)<\alpha(y^*y)(u)$ for all $u\in U$. Choose $h\in B$ positive with norm one such that $\alpha(h)\in \alpha(B)\cong C_0(X)$ has support in $U$. Set $x:=yh^{1/2}$ and observe that
$$
xx^*=yhy^*=\alpha(h)yy^*\in \alpha(B)\quad \text{and}\quad x^*x=h^{1/2}y^*yh^{1/2}=hy^*y\in B,
$$
yielding $\alpha(x^*x)=\alpha(h)\alpha(y^*y)\in \alpha(B)$. By construction of $h$, $0\leq xx^*(u)<\alpha(x^*x)(u)$ for all $u$ in the interior of $\supp(\alpha(h)) \subseteq U$ and $xx^*(u)=\alpha(x^*x)(u)=0$ for all other $u\in X$. In particular, $\Vert xx^*\Vert<\Vert \alpha(x^*x)\Vert=\Vert x^*x\Vert$, which is not possible. Hence, $yy^*=\alpha(y^*y)$.
\end{proof}

As we shall see next, partial isometries occurring in the polar decomposition of elements in $F_{B,\alpha}$ implement $\alpha$ locally.

\begin{lemma}
\label{locallyStronlgyInner}
Let $A$ be a C*-algebra, $B \subseteq A$ a masa, $\alpha\in \Aut(A)$, and $y\in F_{B,\alpha}$. Let $y=v|y|$ be the polar decomposition in $A^{**}$. Then the $*$-homomorphism given by adjoining with $v$,
$$
\Ad(v): \overline{y^*yBy^*y} \to \overline{\alpha(y^*yBy^*y)},
$$
coincides with the restriction of $\alpha$.
\end{lemma}
\begin{proof}
Let $b\in B$ be a positive element. Then 
$$
\Ad(v)(y^*yby^*y)=y|y|b|y|y^*=y|y|b^{1/2}(y|y|b^{1/2})^*.
$$
Now $|y|b^{1/2}\in B$, and therefore $y|y|b^{1/2}\in F_{B,\alpha}$. Using Lemma \ref{automaticAlphaNormality} and the fact that $y^*y\in B$, we conclude that
$$
\Ad(v)(y^*yby^*y)=\alpha((y|y|b^{1/2})^*y|y|b^{1/2})=\alpha(b^{1/2}(y^*y)^2b^{1/2})=\alpha(y^*yby^*y).
$$
\end{proof}

If $v\in F_{B,\alpha}$ happens to be a partial isometry, then the $*$-homomorphism in Lemma~\ref{locallyStronlgyInner} is implemented by $v\in A$ and $\alpha$ restricts to $\Ad(v):v^*vBv^*v \to vv^*\alpha(B)vv^*$. In the light of Lemma~\ref{locallyStronlgyInner}, it is therefore interesting to know whether $F_{B,\alpha}$ contains non-trivial partial isometries. We present the following sufficient criterion.

\begin{lemma}
\label{existencePartialIsometries}
Let $A$ be a $C^*$-algebra, $B \subseteq A$ a masa, $\alpha\in \Aut(A)$, and let $X = \Spec(B)$. Let $x\in F_{B,\alpha}$ and denote by $f \in C_0(X)$ the positive function corresponding to $x^*x\in B \cong C_0(X)$. 

For every clopen subset $V$ of the open support of $f$, there exists a partial isometry $v\in F_{B,\alpha}$ such that $v^*v$ is the characteristic function of $V$.
\end{lemma}
\begin{proof}
We may assume that $V$ is non-empty. Let $e \in B$ denote the characteristic function of $V$, and set
$$
g(t):=
\begin{cases}
f(t)^{-1/2}& ,\ \text{if}\ t\in V,\\
0 & ,\ \text{otherwise.}
\end{cases}
$$
Then $g\in B$ is positive, $y:=xg\in F_{B,\alpha}$ and  satisfies $y^*y=gx^*xg=ef^{-1/2}ff^{-1/2}e=e$.
\end{proof}

As an immediate consequence, we have the following important observation.

\begin{corollary}
\label{TotallyDisconnected}
Let $A$ be a $C^*$-algebra, $B \subseteq A$ a masa, and $\alpha \in \Aut(A)$. Assume that the spectrum of $B$ is totally disconnected.  If $F_{B,\alpha} \neq 0$, then $F_{B,\alpha}$ contains a non-trivial partial isometry.
\end{corollary}

\bdefin
Recall that an $\alpha$-\emph{cocycle} for a discrete group action $\alpha:\Gamma\curvearrowright A $ is a map $u:\Gamma\to \cU(\cM(A))$ satisfying $u_{\gamma \gamma'} = u_\gamma\alpha_\gamma(u_{\gamma'})$ for all $\gamma, \gamma' \in \Gamma$. Given an $\alpha$-cocycle $u$, we define the perturbed $\Gamma$-action $\alpha^u$ via $\alpha^u_\gamma = \Ad(u_\gamma)\circ\alpha_\gamma$. Two actions $\alpha$ and $\beta$ are called \emph{exterior equivalent}, if there exists an $\alpha$-cocycle $u$ such that $\beta=\alpha^u$. Two actions $\alpha$ and $\beta$ are called \emph{cocycle conjugate}, if there is an $\alpha$-cocycle $u$ such that $\alpha^u$ and $\beta$ are conjugate.
\edefin
The crossed products $A\rtimes_\alpha \Gamma$ and $A\rtimes_{\alpha^u} \Gamma$ are isomorphic via an isomorphism fixing $A$ pointwise.

Next we show that if $\alpha$ is an automorphism with $\alpha^n = \id$ and $B \subseteq A$ is a masa with totally disconnected spectrum, then partial isometries in $F_{B,\alpha}$ can be chosen to satisfy a certain ``local cocycle'' condition. Observe that a unitary $u\in \cU(\cM(A))$ satisfying $u\alpha(u)\ldots\alpha^{n-1}(u)=1$ gives rise to an $\alpha$-cocycle (where $\alpha$ is considered as a $\Zz_n$-action) via $u_0 := 1$ and $u_k:=u\alpha(u)\ldots\alpha^{k-1}(u)$ for $k=1,\ldots,n-1$.

\begin{proposition}
\label{localCocycle}
Let $A$ be a $C^*$-algebra, $B \subseteq A$ a masa, and $\alpha:\Zz_n\curvearrowright A$ an action. Assume that $B$ has totally disconnected spectrum. Given a partial isometry $v\in F_{B,\alpha}$, there exists another partial isometry $w\in F_{B,\alpha}$ with $w^*w=v^*v$ and $\alpha^{n-1}(w)\ldots\alpha(w)=w^*$.

In particular, if $A$ is unital and $v$ is a unitary, then $w^*$ is an $\alpha$-cocycle and the perturbed action $\alpha^{w^*}$ satisfies $\alpha^{w^*}(b) = b$ for all $b \in B$.
\end{proposition}
\begin{proof}
For $b\in B$,
$$
\alpha^{n-1}(v)\ldots\alpha(v)vb 
  = \alpha^{n-1}(v)\ldots\alpha(v)\alpha(b)v 
  = \alpha^n(b)\alpha^{n-1}(v)\ldots \alpha(v)v
  = b\alpha^{n-1}(v)\ldots \alpha(v)v.
$$
As $B \subseteq A$ is a masa, this shows that $\alpha^{n-1}(v)\ldots\alpha(v)v\in B$. Using Lemma \ref{automaticAlphaNormality} $n$ times, we compute
\bglnoz
  && \alpha^{n-1}(v)\ldots\alpha(v)vv^*\alpha(v)^*\ldots  \alpha^{n-1}(v)^*
  = \alpha^{n-1}(v)\ldots\alpha(v)\alpha(v^*v)\alpha(v)^*\ldots\alpha^{n-1}(v)^* \\
  &=& \alpha^{n-1}(v)\ldots\alpha^2(v)\alpha(vv^*)\alpha^2(v)^*\ldots\alpha^{n-1}(v)^* = \alpha^{n-1}(vv^*) = v^*v.
\eglnoz
Thus, $f:=(\alpha^{n-1}(v)\ldots\alpha(v)v)^*$ is a local unitary with source and range projection $p=v^*v$. Since $\Spec(B)$ is totally disconnected, we find $b \in B \cong C_0(\Spec(B))$ with $f = b^n$. Then $w:=vb\in F_{B,\alpha}$ satisfies
$$
w^*w=b^*v^*vb=b^*pb=p.
$$
By repeatedly using the defining relation of $F_{B,\alpha}$, we also get that
$$
\alpha^{n-1}(w)\ldots\alpha(w)w=\alpha^{n-1}(v)\ldots\alpha(v)vb^n=f^*f=p.
$$
By Lemma \ref{automaticAlphaNormality},
$$
\alpha^{n-1}(w)\ldots\alpha(w) = \alpha^{n-1}(w)\ldots\alpha(w)\alpha(w^*w) = \alpha^{n-1}(w)\ldots\alpha(w)ww^* = pw^* = w^*,
$$
and the proof is complete.
\end{proof}

Lemma \ref{automaticAlphaNormality} yields that $v,w\in F_{B,\alpha}$ as in Proposition \ref{localCocycle} automatically satisfy
$$
ww^*=\alpha(w^*w)=\alpha(v^*v)=vv^*.
$$

The following observation is crucial for the proof of this section's main result, but also noteworthy in its own right.

\bprop
\label{existenceUnitaryDiagonalNotMasa}
Let $A$ be a unital C*-algebra and $B \subseteq A$ a masa. Assume that $B$ has totally disconnected spectrum, and that for every non-zero projection $e \in B$, there exists an isometry $s \in A$ such that $sBs^* \subseteq eBe$. Moreover, let $\alpha$ be an automorphism on $A$ with the property that $F_{B,\alpha}\neq 0$. 

Then there exists a unitary $w \in A$ satisfying $\alpha \vert_B = \Ad(w) \vert_B$.
\eprop
\begin{proof}
By Corollary \ref{TotallyDisconnected}, there exists a non-trivial partial isometry $v \in F_{B,\alpha}$. Set $e \defeq v^*v$. Then $e \in B$, and we have $\alpha(e) = \alpha(v^*v) = vv^*$ by Lemma~\ref{automaticAlphaNormality}. Choose an isometry $s \in A$ with $sBs^* \subseteq eBe$. Define $w \defeq \alpha(s)^* v s$. Then 
$$
  wbw^* 
  = \alpha(s)^* v (s b s^*) v^* \alpha(s) 
  = \alpha(s)^* \alpha(s b s^*) v v^* \alpha(s) 
  = \alpha(s)^* \alpha(s b s^*) \alpha(e) \alpha(s)
  = \alpha(b)
$$
for every $b \in B$. In particular, $ww^* = 1$. Moreover,
$$
  w^*w 
  = s^* v^* \alpha(s) \alpha(s)^* v s
  = s^* v^* v s s^* s
  = s^* e s = 1.
$$
\end{proof}

\bremark
Note that we can always find isometries as in Proposition~\ref{existenceUnitaryDiagonalNotMasa} if $(A,B) \cong (C^*_r(G),C(G^{(0)}))$ for an {\'e}tale Hausdorff locally compact second countable topologically principal groupoid $G$ which is minimal and purely infinite, and that has the property that $G^{(0)}$ is homeomorphic to the Cantor set, see \cite[Proposition~4.11]{Ma}. In particular, this is the case if $(A,B) \cong (\cO_2,\cD_2)$, where $\cD_2$ is the canonical masa in $\cO_2$, see \cite[\S~III.2]{Rbook} and \cite[Lemma~6.1]{Ma}.
\eremark

\begin{theorem}
\label{characterizationDiagonalNotMasa}
Let $(A,B)$ be as in Proposition~\ref{existenceUnitaryDiagonalNotMasa}. Let $p$ be a prime number and $\alpha:\Zz_p\curvearrowright A$ an action. Then the following are equivalent:
\begin{enumerate}
 \item[(i)]
  $B \subseteq A\rtimes_\alpha \Zz_p$ is not a masa.
  \item[(ii)]
   $\alpha$ is exterior equivalent to a $\Zz_p$-action fixing $B$ pointwise.
\end{enumerate}
Moreover, if $B \subseteq A$ is regular and satisfies the equivalent conditions (i) and (ii), then $B$ is also regular as a sub-C*-algebra of $A \rtimes_\alpha \Zz_p$. 
\end{theorem}
\begin{proof}
Proposition \ref{characterizationMasa} yields that $B \subseteq A\rtimes_\alpha \Zz$ is not a masa if and only if $F_{B,\alpha^k}\neq0$ for some $k\in \gekl{1,\cdots,p-1}$. Since $p$ is prime, we may assume that $k=1$. By Proposition~\ref{existenceUnitaryDiagonalNotMasa} and Proposition~\ref{localCocycle}, there exists an $\alpha$-cocycle $w$ with the property that $\alpha^w(b)=b$ for all $b\in B$. This shows that (i) implies (ii). 

For the other implication, let $\beta:\Zz_p\curvearrowright A$ be an action as in (ii). Then $t  \in A\rtimes_\beta \Zz_p \cap B'$, showing that $B \subseteq A \rtimes_\beta \Zz_p$ is not a masa. As $\alpha$ is exterior equivalent to $\beta$, $A \rtimes_\alpha \Zz_p$ and $A \rtimes_\beta \Zz_p$ are isomorphic via an isomorphism fixing $A$ pointwise. Hence, $B \subseteq A\rtimes_\alpha \Zz_p$ is not a masa.

Lastly, the proof of \an{(ii) implies (i)} also shows that $B$ is regular in $A \rtimes_\alpha \Zz_p$ if it is regular in $A$ and satisfies (i) and (ii).
\end{proof}

\begin{corollary}
\label{UCTregular}
Let $(A,B)$ be a Cartan pair with $A$ nuclear, and assume that $(A,B)$ satisfies the conditions of Proposition~\ref{existenceUnitaryDiagonalNotMasa}. Let $p$ be a prime number and $\alpha:\Zz_p\curvearrowright A$ an action. Assume that there exists an automorphism $\gamma\in \Aut(A)$ such that $\gamma(B)$ is a regular sub-$C^*$-algebra of $A\rtimes_\alpha \Zz_p$. Then $A\rtimes_\alpha \Zz_p$ satisfies the UCT.
\end{corollary}
\begin{proof}
If $\gamma(B) \subseteq A\rtimes_\alpha \Zz_p$ is a masa, then it is a Cartan subalgebra, and $A\rtimes_\alpha \Zz_p$  satisfies the UCT by Theorem~\ref{TwistedGPD-UCT}. Therefore assume that $\gamma(B) \subseteq A\rtimes_\alpha \Zz_p$ is not a masa. Since $p$ is prime, we may assume that $F_{B,\gamma^{-1}\alpha\gamma} = F_{\gamma(B),\alpha}\neq 0$. The conjugate action $\beta:\Zz_p\curvearrowright A$ induced by $\beta:=\gamma^{-1}\alpha\gamma$ has the property that $B \subseteq A\rtimes_\beta\Zz_p$ is not a masa. By Theorem~\ref{characterizationDiagonalNotMasa}, there exists a $\beta$-cocycle $w \in A$ such that $\beta^w$ fixes $B$ pointwise. Thus $A\rtimes_\alpha\Zz_p\cong A\rtimes_{\beta^w}\Zz_p$ satisfies the UCT by Proposition~\ref{Gamma.B=B}.
\end{proof}

As a consequence of Corollary~\ref{UCTregular}, we obtain the following
\bcor
\label{inner-conj-->UCT}
Let $(A,B)$ be a Cartan pair with $A$ nuclear, and assume that $(A,B)$ satisfies the conditions of Proposition~\ref{existenceUnitaryDiagonalNotMasa}. Let $p$ be a prime number and $\alpha:\Zz_p\curvearrowright A$ an action. Suppose that there exists an automorphism $\gamma\in \Aut(A)$ such that $\alpha(\gamma(B)) = u \gamma(B) u^*$ for some $u \in \cU(A)$. Then $A \rtimes_{\alpha} \Zz_p$ satisfies the UCT.
\ecor
\bproof
As $\alpha(\gamma(B)) = u \gamma(B) u^*$, we have that $u^*t \in N_{A \rtimes_\alpha \Zz_p}(\gamma(B))$. Moreover, regularity of $\gamma(B)$ in $A$ yields that $u \in A = C^*(N_A(\gamma(B))) \subseteq C^*(N_{A \rtimes_\alpha \Zz_p}(\gamma(B)))$. Hence, $t = u(u^*t) \in C^*(N_{A \rtimes_\alpha \Zz_p}(\gamma(B)))$. Using again that $\gamma(B)$ is regular in $A$, we conclude that it is also regular in $A \rtimes_\alpha \Zz_p$. The claim now follows from Corollary~\ref{UCTregular}.
\eproof

\bremark
As we mentioned before, Corollary~\ref{UCTregular} and Corollary~\ref{inner-conj-->UCT} in particular apply to $(A,B) = (\cO_2,\cD_2)$. This is already of interest since it suffices to prove the UCT for crossed products of the form $\cO_2 \rtimes_{\alpha} \Zz_p$ in order to solve the UCT problem, see \cite[Theorem 4.17]{BS}. However, also note that by \cite[Corollary 3.8]{CHS}, there exist uncountably many masas in $\cO_n$, $n \in \Nz$, which are outer but not inner conjugate to the canonical masa $\cD_n$.
\eremark

\section{On the UCT problem for $M_{2^\infty}$-stable C*-algebras}
\label{Section UCT}

Let $(G_1,\Sigma_1)$, ..., $(G_n,\Sigma_n)$ be twisted {\'e}tale Hausdorff locally compact second countable topologically principal groupoids, with compact unit spaces $\Sigma_i^{(0)} = G_i^{(0)}$. Set $G \defeq G_1 \times \dotso \times G_n$ as a topological groupoid. Form $\Sigma \defeq \rukl{\Sigma_1 \times \dotso \times \Sigma_n} / \sim$, where $(\dotsc, z \sigma_i, \dotsc, \sigma_j, \dotsc) \sim (\dotsc, \sigma_i, \dotsc, z \sigma_j, \dotsc)$ for all $z \in \Tz$ and $1 \leq i, j \leq n$. Write $[\sigma_1, \dotsc, \sigma_n]$ for the class of $(\sigma_1, \dotsc, \sigma_n)$ in $\Sigma$. It is easy to see that $\Sigma$ becomes a Hausdorff locally compact second countable groupoid via $[\sigma_1, \dotsc, \sigma_n] \cdot [\tau_1, \dotsc, \tau_n] = [\sigma_1 \tau_1, \dotsc, \sigma_n \tau_n]$. Moreover, we have a canonical $\Tz$-action given by $z [\sigma_1, \dotsc, \sigma_n] \defeq [z \sigma_1, \dotsc, \sigma_n] = \dotso = [\sigma_1, \dotsc, z \sigma_n]$. It is also easy to see that the map $\Sigma \to G, \, [\sigma_1, \dotsc, \sigma_n] \ma (\dot{\sigma}_1, \dotsc, \dot{\sigma}_n)$ gives rise to a central groupoid extension
$\Tz \times G^{(0)} \rightarrowtail \Sigma \twoheadrightarrow G$,
so that $(G,\Sigma)$ is again a twisted {\'e}tale Hausdorff locally compact second countable topologically principal groupoid.

\blemma \label{Cartan Tensor Product}
We have a canonical identification $C^*_r(G_1,\Sigma_1) \otimes_{\rm min} \dotso \otimes_{\rm min} C^*_r(G_n,\Sigma_n) \overset{\cong}{\lori} C^*_r(G,\Sigma)$ sending $f_1 \otimes \dotso \otimes f_n$ to the function $[\sigma_1, \dotsc, \sigma_n] \ma f(\sigma_1) \dotsm f(\sigma_n)$ in $C_c(G,\Sigma)$, for $f_i \in C_c(G_i,\Sigma_i)$.
\elemma
\bproof
By induction, it suffices to treat the case $n = 2$. In that case, let $x_1 \in G_1^{(0)}$ and $x_2 \in G_2^{(0)}$ be arbitrary, and let $\pi_{x_1}$, $\pi_{x_2}$ be the representations of $C_c(G_1,\Sigma_1)$, $C_c(G_2,\Sigma_2)$ on the Hilbert spaces $\mathscr{H}_{x_1}$, $\mathscr{H}_{x_2}$ introduced in  Section \ref{Prelim}. Also, let $\pi_{(x_1,x_2)}$ be the representation of $C_c(G,\Sigma)$ on $\mathscr{H}_{(x_1,x_2)}$ as in Section \ref{Prelim}. A straightforward computation shows that $\xi_1 \otimes \xi_2 \ma \eckl{[\sigma_1, \sigma_2] \ma \xi(\sigma_1) \cdot \xi(\sigma_2)}$ extends to a unitary $U: \: \mathscr{H}_{x_1} \otimes \mathscr{H}_{x_2} \overset{\cong}{\lori} \mathscr{H}_{(x_1,x_2)}$. In addition, let $f_i \in C_c(G_i,\Sigma_i)$ and write $f$ for the function $[\sigma_1, \sigma_2] \ma f_1(\sigma_1) \cdot f_2(\sigma_2)$ in $C_c(G,\Sigma)$. Then it is straightforward to check that $U \circ (\pi_{x_1}(f_1) \otimes \pi_{x_2}(f_2)) = (\pi_{(x_1,x_2)}(f)) \circ U$. Our claim follows.
\eproof

Let $(G_i,\Sigma_i)$, $i \in \Nz$, be twisted {\'e}tale Hausdorff locally compact second countable topologically principal groupoids with compact unit spaces $\Sigma_i^{(0)} = G_i^{(0)}$. Form $G \defeq {\prod'}_{i=1}^{\infty} (G_i, G_i^{(0)})$. Here, we define ${\prod'}_{i=1}^{\infty} (G_i, G_i^{(0)}) \defeq \menge{(\gamma_i)_i \in \prod_{i=1}^{\infty} G_i}{\gamma_i \in G_i^{(0)} \ {\rm for \ almost \ all} \ i \in \Nz}$. $G$ becomes an {\'e}tale Hausdorff locally compact second countable topologically principal groupoid with respect to component-wise multiplication. Form $\Sigma \defeq {\prod'}_{i=1}^{\infty} (\Sigma_i, \Sigma_i^{(0)}) / \sim$, where the restricted product is defined as above, and $(\dotsc, z \sigma_i, \dotsc, \sigma_j, \dotsc) \sim (\dotsc, \sigma_i, \dotsc, z \sigma_j, \dotsc)$ for all $z \in \Tz$ and $i, j \in \Nz$. Write $[\sigma_i]_i$ for the class of $(\sigma_i)_i$ in $\Sigma$. It is easy to see that $\Sigma$ becomes a Hausdorff locally compact second countable groupoid via $[\sigma_i]_i \cdot [\tau_i]_i = [\sigma_i \tau_i]_i$. Moreover, we have a canonical $\Tz$-action given by $z [\sigma_1, \sigma_2, \dotsc] \defeq [z \sigma_1, \sigma_2, \dotsc] = [\sigma_1, z \sigma_2, \dotsc] = \dotso$. Obviously, the map $\Sigma \to G, \, [\sigma_i]_i \ma (\dot{\sigma}_i)_i$ yields a central groupoid extension
$\Tz \times G^{(0)} \rightarrowtail \Sigma \twoheadrightarrow G$,
so that $(G,\Sigma)$ is again a twisted {\'e}tale Hausdorff locally compact second countable topologically principal groupoid.

\blemma \label{Cartan Infinite Tensor Product}
We have a canonical identification $\bigotimes_{i=1}^{\infty} C^*_r(G_i,\Sigma_i) \cong C^*_r(G,\Sigma)$.
\elemma
\bproof
In the infinite tensor product $\bigotimes_{i=1}^{\infty} C^*_r(G_i,\Sigma_i)$, there is a canonical dense sub-*-algebra given by the algebraic inductive limit of the sub-*-algebras $\bigotimes_{i=1}^n C^*_r(G_i,\Sigma_i)$, $n \in \Nz$. In $C^*_r(G,\Sigma)$, we also have a canonical dense sub-*-algebra given by the algebraic inductive limit of the sub-*-algebras $C^*_r(G^n,\Sigma^n)$, $n \in \Nz$, where $G^n = G_1 \times \dotso \times G_n$ and $\Sigma^n = \rukl{\Sigma_1 \times \dotso \times \Sigma_n} / \sim$ as above. Using Lemma \ref{Cartan Tensor Product}, the building blocks of these inductive limits can be identified, in a way compatible with the connecting homomorphisms of the inductive limits.
\eproof

In the remainder of this section, we would like to present a reformulation of the UCT problem for separable, nuclear C*-algebras that are $KK$-equivalent to their $M_{2^\infty}$-stabilization in terms of Cartan subalgebras and strongly approximately inner $\Zz_2$-actions of $\cO_2$. Recall from \cite[Definition 3.6]{Iz} that a $\Zz_2$-action $\beta$ on a unital $C^*$-algebra $B$ is called strongly approximately inner if there exist unitaries $u_n \in B$, $n \in \Nz$, such that $\beta(u_n) = u_n$ for all $n \in \Nz$ and $u_nbu_n^*$ converges to $\beta(b)$ for all $b \in B$. To the best of the authors' knowledge, it is not known whether all $\Zz_2$-actions on $\cO_2$ are strongly approximately inner, see also \cite[Remark~4.9 (2)]{Iz}.

We start by recalling Izumi's symmetry $\alpha: \Zz_2 \curvearrowright \cO_2$ from \cite[Lemma 4.7]{Iz}. Let $e \in \cO^{\mathrm{st}}_\infty$ be a projection with the property that $[e] \in K_0(\cO^{\mathrm{st}}_\infty) \cong \Zz$ is a generator. Set $v = 2e - 1$ and fix an isomorphism $\cO_2 \cong \bigotimes_{n \in \Nz} \cO_\infty^{\mathrm{st}}$. We define $\alpha$ to be the unique $\Zz_2$-action on $\cO_2$ satisfying $(\cO_2,\alpha) \cong (\bigotimes_{n \in \Nz} \cO_\infty^{\mathrm{st}}, \bigotimes_{n \in \Nz} \Ad(v))$. Observe that $\alpha$ is strongly approximately inner.

For an action $\beta:\Zz_2 \curvearrowright B$ on a unital C*-algebra, we write $e_\beta \in B \rtimes_ \beta \Zz_2$ for the projection $e_\beta := \frac{1 + t}{2}$, where $t \in B \rtimes_\beta \Zz_2$ is the canonical symmetry implementing $\beta$.

\begin{lemma} \label{e_alpha=1}
We have that $(K_0(\cO_2\rtimes_\alpha \Zz_2),[e_\alpha]) \cong (\Zz[\frac{1}{2}],1)$.
\end{lemma}
\begin{proof}
Izumi has shown in \cite[Lemma 4.7]{Iz} that $K_0(\cO_2\rtimes_\alpha \Zz_2) \cong \Zz[\frac{1}{2}]$. In fact, $\cO_2 \rtimes_\alpha \Zz_2 \cong \cO_\infty^{\mathrm{st}} \otimes M_{2^\infty}$. Let $e$ and $v$ be as above, and let $\tilde{t}\in \cO_\infty^{\mathrm{st}}\rtimes_{\Ad(v)} \Zz_2$ and $t \in \cO_2\rtimes_\alpha \Zz_2$ denote the canonical symmetries implementing $\Ad(v)$ and $\alpha$, respectively. Observe that $\frac{1 + v\tilde{t}}{2}$ and $\frac{1 - v\tilde{t}}{2}$ are central orthogonal projections in $\cO_\infty^{\mathrm{st}}\rtimes_{\Ad(v)} \Zz_2$ with the property that
$$
\begin{array}{c}
\cO_\infty^{\mathrm{st}} \rtimes_{\Ad(v)} \Zz_2 = \cO_\infty^{\mathrm{st}} \rtimes_{\Ad(v)} \Zz_2 \left( \frac{1 + v\tilde{t}}{2} \right) \oplus \cO_\infty^{\mathrm{st}} \rtimes_{\Ad(v)} \Zz_2 \left( \frac{1 - v\tilde{t}}{2} \right) \cong \cO_\infty^{\mathrm{st}} \oplus \cO_\infty^{\mathrm{st}}.
\end{array}
$$
From this we conclude that $K_0(\cO_\infty^{\mathrm{st}}\rtimes_{\Ad(v)} \Zz_2) \cong \Zz^2$ with generators $\left[e \left(  \frac{1 + v\tilde{t}}{2} \right)\right]$ and $\left[e \left(  \frac{1 - v\tilde{t}}{2} \right)\right]$.
We compute
$$
\begin{array}{lclll}
e_{\Ad(v)} & = & \frac{1}{2}\left( \frac{1 + v\tilde{t} + 1 - v\tilde{t}} {2} + v^2 \tilde{t}\right)
 & \hspace{-1.7cm} = & \hspace{-1.7cm} \frac{1}{2} \left( (1 + v) \left( \frac{1 + v\tilde{t}}{2} \right) + (1 - v) \left( \frac{1 - v\tilde{t}}{2} \right) \right) \\
  & = & \left( \frac{1 + v}{2} \right) \left( \frac{1 + v\tilde{t}}{2} \right) + \left( \frac{1 - v}{2} \right) \left( \frac{1 - v\tilde{t}}{2} \right)
 & = & e \left(\frac{1 + v\tilde{t}}{2} \right) + (1 - e)\left( \frac{1 - v\tilde{t}}{2} \right),
\end{array}
$$
and hence 
$$
\begin{array}{c}
\left[e_{\Ad(v)}\right] = \left[e \left(\frac{1 + v\tilde{t}}{2} \right)\right] - \left[e\left(\frac{1 - v\tilde{t}}{2}\right)\right]\in K_0(\cO_\infty^{\mathrm{st}}\rtimes_{\Ad(v)} \Zz_2).
\end{array}
$$
The canonical $*$-homomorphism $\cO_\infty^{\mathrm{st}}\rtimes_{\Ad(v)} \Zz_2 \to \cO_2 \rtimes_\alpha \Zz$ clearly maps $e_{\Ad(v)}$ to $e_\alpha$. It now follows from the proof of \cite[Lemma 4.7]{Iz} that $(K_0(\cO_2 \rtimes_\alpha \Zz_2), [e_\alpha]) \cong (\Zz[\frac{1}{2}], 1)$.
\end{proof}

Note that at least a priori, the construction of $\alpha$ depends on the choice of $e$. However, combining Lemma \ref{e_alpha=1} with Izumi's classification of outer strongly approximately inner $\Zz_2$-actions on $\cO_2$ (see \cite[Theorem 4.8]{Iz}), we conclude that any other choice would lead to an action that is conjugate to $\alpha$. We also get the following characterization of outer strongly approximately inner $\Zz_2$-actions on $\cO_2$ whose associated crossed products satisfy the UCT.

\bprop \label{all crossed products}
Let $\beta: \Zz_2 \curvearrowright \cO_2$ be an outer strongly approximately inner action with the property that $\cO_2 \rtimes_\beta \Zz_2$ satisfies the UCT. Then there exists a unital, $M_{2^\infty}$-absorbing UCT Kirchberg algebra $A$ and an isomorphism $A \otimes \cO_2 \cong \cO_2$ such that $(A \otimes \cO_2, \id_A \otimes \alpha) \cong (\cO_2, \beta)$.
\eprop
\bproof
Let $A$ be the unique unital UCT Kirchberg algebra with 
$$
(K_0(A),[1],K_1(A)) \cong (K_0(\cO_2\rtimes_\beta \Zz_2),[e_\beta],K_1(\cO_2\rtimes_\beta \Zz_2)),
$$
which exists by \cite[Theorem~3.6]{Ror} (uniqueness follows from Kirchberg-Phillips classification \cite{Kir,Phi}). The canonical isomorphism $(A \otimes \cO_2) \rtimes_{\id_A \otimes \alpha} \Zz_2 \stackrel{\cong}{\longrightarrow} A \otimes (\cO_2 \rtimes_\alpha \Zz_2)$ maps $e_{\id_A \otimes \alpha}$ to $1 \otimes e_\alpha$. By the K{\"u}nneth theorem and Lemma \ref{e_alpha=1}, the induced isomorphism satisfies
\[
\begin{array}{lcl}
(K_0((A \otimes \cO_2) \rtimes_{\id_A \otimes \alpha} \Zz_2),[e_{\id_A\otimes \alpha}]) & \cong & (K_0(A) \otimes K_0(\cO_2 \rtimes_\alpha \Zz_2),[1]\otimes [e_\alpha]) \\
 & \cong & (K_0(A) \otimes \Zz[\frac{1}{2}],[1]\otimes 1) \\
 & \cong & (K_0(\cO_2\rtimes_\beta \Zz_2) \otimes \Zz[\frac{1}{2}],[e_\beta]\otimes 1) \\
 & \cong & (K_0(\cO_2\rtimes_\beta \Zz_2),[e_\beta]).
\end{array}
\]
The last isomorphism comes from the fact that $\cO_2 \rtimes_\beta \Zz_2$ is $M_{2^\infty}$-absorbing, which in turn follows from \cite[Lemma~4.4]{Iz} and Kirchberg-Phillips classification. As both $(A \otimes \cO_2) \rtimes_{\id_A \otimes \alpha} \Zz_2$ and $\cO_2\rtimes_\beta \Zz_2$ are in Cuntz standard form (see \cite[proof of Lemma~4.4]{Iz}) and satisfy the UCT, Kirchberg-Phillips classification yields an isomorphism $\phi:(A \otimes \cO_2) \rtimes_{\id_A \otimes \alpha} \Zz_2 \stackrel{\cong}{\longrightarrow} \cO_2\rtimes_\beta \Zz_2$ with $K_0(\phi)([e_{\id_A\otimes \alpha}]) = [e_\beta]$. Hence, by \cite[Theorem 4.8 (1)]{Iz}, $\beta$ is conjugate to $\id_A\otimes \alpha$.
\eproof

We draw the following consequence, which, to the best of the authors' knowledge, is not documented in the literature.

\bcor
There exist outer strongly approximately inner $\Zz_2$-actions on $\cO_2$ that are cocycle conjugate but not conjugate.
\ecor
\bproof
Take unital, $M_{2^\infty}$-absorbing UCT Kirchberg algebras $A$ and $B$, which are $KK$-equivalent but not isomorphic (i.e. have different positions of the unit). Then it follows from the proof of Proposition \ref{all crossed products} and \cite[Theorem 4.8 (1)+(2)]{Iz} that $\id_A\otimes\alpha$ and $\id_B\otimes \alpha$ are cocycle conjugate but not conjugate.
\eproof

Let us come back to the $\Zz_2$-action $\alpha$ on $\cO_2$. We now show, by carefully choosing the projection $e$, that it fixes a Cartan subalgebra of $\cO_2$ pointwise.

\bprop \label{alpha Cartan fixed}
There exists a Cartan subalgebra $B \subseteq \cO_2$ such that $\alpha$ fixes $B$ pointwise. Moreover, $B$ can be chosen to be isomorphic to the algebra of continuous functions on the Cantor set.
\eprop
\bproof
Let $\cD_\infty = \mathrm{C}^*(S_\mu S^*_\mu, 1 ~:~ \mu \in W) \subseteq \cO_\infty$, where $W$ denotes the set of finite words in $\Nz$. Set $p := 1 - S_1S_1^*$ and $e := S_2S_2^*$ and note that $e,p \in \cD_\infty$. As $[p] = 0 \in K_0(\cO_\infty)$, $p\cO_\infty p \cong \cO^{\mathrm{st}}_\infty$. Of course, $e \leq p$ and therefore $e \in p\cO_\infty p$. Denote by $v := 2e - p \in p \cD_\infty p$. The natural map  $K_0(p\cO_\infty p) \to K_0(\cO_\infty)$ sends $[e]$ to $[e] = [1]\in K_0(\cO_\infty)$, and therefore $[e] \in K_0(p\cO_\infty p) \cong \Zz$ must be a generator. Thus, $\alpha$ is conjugate to $\bigotimes_{\Nz} \Ad(v): \Zz_2\curvearrowright  \bigotimes_{\Nz} p\cO_\infty p$. Moreover, as $v \in p\cD_\infty p$, $\Ad(v)$ fixes $p \cD_\infty p$ pointwise. It follows that the action $\bigotimes_{\Nz} \Ad(v)$ fixes $\bigotimes_{\Nz} p\cD_\infty p$ pointwise.  Note also that $(p\cO_\infty p, p\cD_\infty p)$ is a Cartan pair, as this is true for $(\cO_\infty, \cD_\infty)$, see \cite[\S~III.2]{Rbook}. The claim now follows from Lemma \ref{Cartan Infinite Tensor Product}.
\eproof

\btheo
\label{all automorphisms Cartan fixed}
Let $\beta:\Zz_2\curvearrowright \cO_2$ be an outer strongly approximately inner action. $\cO_2 \rtimes_\beta \Zz_2$ satisfies the UCT if and only if there exists a Cartan subalgebra $B \subseteq \cO_2$ which is fixed by $\beta$ pointwise.
\etheo
\bproof
Assume that $\cO_2 \rtimes_\beta \Zz_2$ satisfies the UCT. By Proposition \ref{alpha Cartan fixed}, $\alpha$ fixes some Cartan subalgebra $D \subseteq \cO_2$ pointwise.  By Proposition \ref{all crossed products}, there exists a unital, $M_{2^\infty}$-absorbing UCT Kirchberg algebra $A$ such that $\beta$ is conjugate to $\id_A\otimes \alpha$. As $A$ satisfies the UCT, it admits a Cartan subalgebra $C\subseteq A$, see Remark \ref{UCT <-> Kirchberg Cartan}. Then by Lemma \ref{Cartan Tensor Product}, $C\otimes D \subseteq A\otimes \cO_2$ is a Cartan subalgebra that gets fixed by $\id_A\otimes \alpha$ pointwise. This shows that there exists a Cartan subalgebra $B \subseteq \cO_2$ that is fixed by $\beta$ pointwise.

The reverse implication follows from Proposition~\ref{Gamma.B=B}.
\eproof

\btheo
The following statements are equivalent:
\begin{enumerate}
\item[(i)] Every separable, nuclear C*-algebra $A$ that is $KK$-equivalent to $A \otimes M_{2^\infty}$ satisfies the UCT.
\item[(ii)] For every unital, $M_{2^\infty}$-absorbing Kirchberg algebra $A$ in Cuntz standard form, there exists a strongly approximately inner action $\beta:\Zz_2 \curvearrowright \cO_2$ and a Cartan subalgebra $B \subseteq \cO_2$ such that $\cO_2 \rtimes_\beta \Zz_2 \cong A$ and $\beta$ fixes $B$ pointwise.
\item[(iii)] Every outer strongly approximately inner $\Zz_2$-action on $\cO_2$ fixes some Cartan subalgebra $B \subseteq \cO_2$ pointwise.
\item[(iv)] Every outer strongly approximately inner $\Zz_2$-action on $\cO_2$ fixes some Cartan subalgebra $B \subseteq \cO_2$ globally.
\end{enumerate}
\etheo
\bproof
Assume that (i) holds and $A$ is as in (ii). Then $A$ is a UCT Kirchberg algebra, and therefore admits a Cartan subalgebra by Remark \ref{UCT <-> Kirchberg Cartan}. As $A$ is in Cuntz standard form and absorbs $M_{2^\infty}$ tensorially, we conclude that $A \cong  (A \otimes \cO_2) \rtimes_{\id_A \otimes \alpha} \Zz_2$. Fix an isomorphism $\cO_2 \cong A \otimes \cO_2$, which exists by Kirchberg's absorption theorem \cite{KP}, and let $\beta$ denote the action corresponding to $\id_A \otimes \alpha$ under this identification. By Theorem~\ref{all automorphisms Cartan fixed}, $\beta$ then has the desired property and we conclude that (i) implies (ii).

Now assume (ii) and let $\beta$ be an outer strongly approximately inner $\Zz_2$-action on $\cO_2$. Then $\cO_2 \rtimes_\beta \Zz_2$ is a unital, $M_{2^\infty}$-absorbing Kirchberg algebra in Cuntz standard form. By (ii) and Proposition~\ref{Gamma.B=B}, $\cO_2 \rtimes_\beta \Zz_2$ therefore satisfies the UCT. Claim (iii) now follows from Theorem~\ref{all automorphisms Cartan fixed}.

The implication from (iii) to (iv) is trivial.

Lastly, assume (iv) and let $A$ be a separable, nuclear C*-algebra with the property that it is $KK$-equivalent to $A \otimes M_{2^\infty}$. Then $A$ is $KK$-equivalent to a unital, $M_{2^{\infty}}$-absorbing Kirchberg algebra in Cuntz standard form $A'$ by \cite[Theorem I]{Kir}. Fix an isomorphism $A' \otimes \cO_2 \cong \cO_2$ and let $\beta$ be the unique automorphism corresponding to $\id_{A'} \otimes \alpha$ under this isomorphism. Then $\beta$ is outer strongly approximately inner and satisfies
\[
\cO_2 \rtimes_\beta \Zz_2 \cong (A' \otimes \cO_2) \rtimes_{\id \otimes \alpha} \Zz_2 \cong A' \otimes \cO_\infty^{\mathrm{st}} \otimes M_{2^\infty} \cong A'.
\]
Thus, there exists an outer strongly approximately inner action $\beta$ on $\cO_2$ such that $A$ is $KK$-equivalent to $\cO_2 \rtimes_\beta \Zz_2$. Using Proposition~\ref{Gamma.B=B}, we therefore conclude that (i) holds. This concludes the proof.
\eproof

\end{document}